\documentclass{amsart}

\usepackage[backend=biber,maxbibnames=5,maxalphanames=5,style=alphabetic,bibencoding=utf8,giveninits,url=false,isbn=false]{biblatex}
\AtBeginBibliography{\small}

\usepackage{verse}
\usepackage{float}

\newlength\lineindent

\usepackage[all,hyperref]{bi-discrete}
\usepackage{cleveref}
\usepackage{tikz}
\usetikzlibrary{calc}

\usepackage{mathtools}
\usepackage{algorithm2e}

\DeclareMathSymbol{\shortminus}{\mathbin}{AMSa}{"39}

\providecommand{\tria}{\mathcal{T}}

\providecommand{\Wdiv}{W_q(\textup{div},\Omega)}
\providecommand{\WdivTwo}{W_2(\textup{div},\Omega)}

\providecommand{\Wdivf}{W_q(\textup{div}{=}{\shortminus}f,\Omega)}
\providecommand{\WdivTwof}{W_2(\textup{div}{=}{\shortminus}f,\Omega)}

\providecommand{\dx}{\, \mathrm{d}x}

\providecommand{\snorm}[1]{{\left\vert\kern-0.25ex\left\vert\kern-0.25ex\left\vert #1   \right\vert\kern-0.25ex\right\vert\kern-0.25ex\right\vert}}

\usepackage{dsfont}
\usepackage{pgfplots}
\usepackage{tikz}
\usepackage{listings}
\usepackage{stmaryrd}

\usepackage{tikz}
\usepackage{pgfplots}
\pgfplotsset{
  width=.65\linewidth,
  axis background/.style={fill=black!5!white},
  grid style={densely dotted,semithick},
  legend style={
    legend columns=1,
    legend pos=outer north east
  },
  compat=newest 
}
\pgfplotscreateplotcyclelist{MyColors}{%
    {red,mark = *,every mark/.append style={solid,scale=0.4,fill=red}},
    {teal,mark = x,every mark/.append style={solid,scale=0.4,fill=teal}},
    {blue,mark = square*,every mark/.append style={solid,scale=0.4,fill=blue}},
{dotted,teal,mark = *,every mark/.append style={solid,scale=0.4,fill=red}},
    {dotted,teal,mark = x,every mark/.append style={solid,scale=0.4,fill=teal}},
    {dotted,blue,mark = square*,every mark/.append style={solid,scale=0.4,fill=blue}},
{dash dot,red,mark = *,every mark/.append style={solid,scale=0.4,fill=red}},
    {dash dot,teal,mark = x,every mark/.append style={solid,scale=0.4,fill=teal}},
    {dash dot,blue,mark = square*,every mark/.append style={solid,scale=0.4,fill=blue}},}

\pgfplotscreateplotcyclelist{MyColors3}{%
{red,mark = *,every mark/.append style={solid,scale=0.4,fill=red}},
{teal,mark = square*,every mark/.append style={solid,scale=0.4,fill=teal}},
{blue, mark=diamond*,every mark/.append style={solid,scale=0.4,fill=blue}},
{red,mark = *,dashed,every mark/.append style={solid,scale=0.4,fill=white}},
{teal,dashed,mark = square*,every mark/.append style={solid,scale=0.4,fill=white}},
{blue,dashed,mark=diamond*,every mark/.append style={solid,scale=0.5,fill=white}},}

\makeatletter
\@namedef{subjclassname@2020}{%
  \textup{2020} Mathematics Subject Classification}
\makeatother

\begin{document}
\lstset{language=Python,
basicstyle=\small, 
keywordstyle=\color{black}\bfseries, 
commentstyle=\color{blue}, 
stringstyle=\ttfamily, 
showstringspaces=false,
numbers=left, 
numberstyle=\small, 
numbersep=10 pt,
xleftmargin= 27pt,
}

\author[A.~Kh.~Balci]{Anna Kh.~Balci}
\author[L.~Diening]{Lars Diening}
\author[J.~Storn]{Johannes Storn}
\address[A.~Balci, L.~Diening, J.~Storn]{Department of Mathematics, University of Bielefeld, Postfach 10 01 31, 33501 Bielefeld, Germany}
\email{akhripun@math.uni-bielefeld.de}
\email{lars.diening@uni-bielefeld.de}
\email{jstorn@math.uni-bielefeld.de}


\thanks{The work of the authors was supported by the Deutsche Forschungsgemeinschaft (DFG, German Research Foundation) – SFB 1283/2 2021 – 317210226.}

\title{Relaxed Ka\v{c}anov scheme for the $p$-Laplacian with large $p$}

\begin{abstract}
We introduce a globally convergent relaxed \Kacanov{} scheme for the computation of the discrete minimizer to the $p$-Laplace problem with $2 \leq p < \infty$. 
The iterative scheme is easy to implement since each iterate results only  from the solve of a weighted, linear Poisson problem. It neither requires an additional line search nor involves unknown constants for the step length. The rate of convergence is independent of the underlying mesh. 
\end{abstract}
\maketitle
\section{Introduction}
The $p$-Laplace problem is a prototype of many non-linear problems occurring in simulations of non-Newtonian fluids, turbulent flows of gases, glaciology, and plastic modeling. Given a right-hand side $f \in L^q(\Omega)$ with $1/p + 1/q = 1$ and $p\in (1,\infty)$, the $p$-Laplace problem seeks the unique minimizer
\begin{align}\label{eq:PrimalProblemIntro}
u = \argmin_{v \in W_0^{1,p}(\Omega)}\, \mathcal{J}( v ) \qquad\text{with }\mathcal{J}(v) \coloneqq
 \frac{1}{p}\int_\Omega |\nabla v|^p \,\mathrm{d}x - \int_\Omega f v\,\mathrm{d}x.
\end{align}
The functions may be scalar or vector-valued, but for better readability we use the notation of the scalar-valued case.  One difficulty that arises in its numerical approximation is the computation of the discretized minimizer.  In fact, established iterative schemes like Newton or gradient descent methods experience huge instabilities or even fail for values of $p$ that are not close to two, see Section~\ref{subsec:Exp3}.  A numerical scheme that overcomes these difficulties for small values of $1< p \leq 2$ is the relaxed \Kacanov{} scheme in \cite{DieningFornasierWank17} (see also \cite{BartelsDieningNochetto18} for an alternative approach using the $p$-Laplace gradient flow covering the range $1\leq p < 2$). However, this scheme does not converge for $p > 3$, see \cite[Rem.~21]{DieningFornasierWank17}. We modify this approach and suggest a novel algorithm to compute the discrete minimizer of the $p$-Laplacian for large values of $p \geq 2$, e.g.\ $p=100$ in Section~\ref{subsec:Exp2}.  The resulting iterative scheme
\begin{itemize}
\item converges globally for all $p\geq 2$,
\item is cheap and easy to implement, since each iteration computes solely the Galerkin approximation to a weighted Poisson model problem and avoids any additional line search,
\item includes a regularization that ensures well-posedness of the iterative computations,
\item is robust with respect to (adaptive) mesh refinements,
\item allows for a fully adaptive scheme including adaptive mesh refinements and adaptions of the regularization parameters,
\item applies to the scalar and vector-valued case.
\end{itemize}
These advantages make our \Kacanov{} scheme very attractive for the computation of $p$-Laplace problems with large values of $p$ as for example needed in \cite{BouchitteButtazzoDePascale03,BarrettPrigozhin00}.

We design the \Kacanov{} scheme as follows.
Section~\ref{sec:RelaxedDualFunctional} introduces the dual problem seeking the minimizer $\sigma \in \Wdivf \coloneqq \lbrace \tau \in L^q(\Omega;\mathbb{R}^d)\mid \textup{div}\, \tau = -f\rbrace$ with
\begin{align}\label{eq:dualProblemIntro}
\sigma = \argmin_{\tau \in \Wdivf} \mathcal{J}^*(\tau)\qquad\text{where }\mathcal{J}^*(\tau) \coloneqq \frac{1}{q}\,\int_\Omega |\tau|^q\,\mathrm{d}x.
\end{align}
The dual problem allows for a similar relaxation of the energy functional as in \cite{DieningFornasierWank17}.
We prove the convergence of the relaxed dual functional with respect in the relaxation parameter in Section~\ref{sec:conInRelaxPara}, provided the solution has the additional smoothness $\sigma \in L^2(\Omega;\mathbb{R}^d)$ (see Theorem~\ref{thm:maxRegularity} for a discussion on that regularity assumption). We show the convergence of the \Kacanov{} iterations for fixed relaxation parameters in Section~\ref{sec: Kacanov} and combine the two convergence results to conclude an algebraic rate of convergence towards the exact discrete minimizer in Section~\ref{sec:AlgebraicRate}.
Section~\ref{sec:Discretization} introduces a discretization of the \Kacanov{} scheme that leads due to a duality relation on the discrete level to a numerical scheme that computes in each iterate solely the Galerkin approximation to a weighted Poisson model problem with lowest-order Lagrange elements. Section~\ref{sec:adaptScheme} suggests an adaptive scheme that estimates the regularization error, the error in the \Kacanov{} iterations, and the discretization error. Depending on these estimates it either causes an adaption of the regularization parameter, computes a further \Kacanov{} iteration, or applies an adaptive mesh refinement. The numerical experiments in Section~\ref{subsec:Exp1} and~\ref{subsec:Exp2} indicate even for large values $p=100$ an exponential rate of convergence for that adaptive approach. The numerical experiment in Section~\ref{subsec:Exp3} compares our scheme with the steepest descent method from \cite{HuangLiLiu07}. It illustrates that our scheme does, unlike the steepest descent method, not depend on the underlying triangulation and is thus superior on fine meshes.

\section{Relaxed dual functional}\label{sec:RelaxedDualFunctional}
The equivalence of the primal problem in \eqref{eq:PrimalProblemIntro} and its dual formulation in \eqref{eq:dualProblemIntro} is a classical result in the calculus of variations, see for example \cite{EkelandTemam76}. It follows for a wide class of convex functionals by properties of the convex conjugate. 
The Euler--Lagrange equation of \eqref{eq:dualProblemIntro} involves the spaces
\begin{align}\label{eq:def_spaces}
\begin{aligned}
\Wdiv & \coloneqq  \lbrace \tau \in L^q(\Omega;\mathbb{R}^d) \mid \divergence \tau \in L^q(\Omega)\rbrace,\\
\end{aligned}
\end{align}
The minimizer $\sigma \in \Wdiv$ in \eqref{eq:dualProblemIntro} solves, with a Lagrange multiplier $u\in L^p(\Omega)$ the saddle point problem
\begin{align}\label{eq:dualProbVar}
\begin{aligned}
\int_\Omega |\sigma|^{q-2} \sigma \cdot \xi \,\mathrm{d}x + \int_\Omega u\,\textup{div}\,\xi\,\mathrm{d}x & = 0 &&\text{for all }\xi \in \Wdiv,\\
\int_\Omega v\,\textup{div}\, \sigma \,\mathrm{d}x &= - \int_\Omega f v \,\mathrm{d}x&&\text{for all }v\in L^p(\Omega). 
\end{aligned}
\end{align}
The Lagrange multiplier $u\in L^p(\Omega)$ equals the minimizer in \eqref{eq:PrimalProblemIntro} and satisfies $\sigma = |\nabla u|^{p-2}\nabla u$ and $\nabla u = |\sigma|^{q-2} \sigma$. Moreover, we have the identity
\begin{align}\label{eq:dualityRelation1}
-\mathcal{J}^*(\sigma) = \mathcal{J}(u).
\end{align}

The problem in \eqref{eq:dualProbVar} has similarities to the mixed formulation of a weighted Poisson model problem and so motivates an iterative calculation of functions $\sigma_{n+1} \in \Wdiv$ and $u_{n+1}\in  L^p(\Omega)$ by solving, for all $\xi \in \Wdiv$ and $v \in L^p(\Omega)$,
\begin{align}\label{eq:UnrelaxedKac}
\begin{aligned}
\int_\Omega |\sigma_n|^{q-2} \sigma_{n+1} \cdot \xi \,\mathrm{d}x + \int_\Omega u_{n+1}\,\textup{div}\,\xi\,\mathrm{d}x & = 0,\\
\int_\Omega v\,\textup{div}\, \sigma_{n+1} \,\mathrm{d}x &= - \int_\Omega f v \,\mathrm{d}x. 
\end{aligned}
\end{align}
This approach (known as \Kacanov\ iteration, Picard iteration, or method of successive substitutions) has already been suggested in \cite{CreuseFarhloulPaquet07}. Unfortunately, the problem in \eqref{eq:UnrelaxedKac} degenerates at points where $|\sigma| = 0$ and $|\sigma| = \infty$. We remedy this difficulty with an idea from \cite{DieningFornasierWank17}. This idea bases on a modified energy functional which reads for all $a\colon \Omega \to [0,\infty)$ and  $\tau \in \Wdiv$
\begin{align}\label{eq:relaxedFunctional}
\mathcal{J}^*(\tau,a) \coloneqq \int_\Omega \frac{1}{2} a^{q-2}|\tau|^2 + \left(\frac{1}{q} - \frac{1}{2}\right)a^q\,\mathrm{d}x. 
\end{align} 
Notice that for $a = |\tau|$ the energy equals the one in \eqref{eq:dualProblemIntro}. 
Moreover, the relaxed energy is convex with respect to $\tau$ and $a$, which we show as follows.
We set the function $\beta(t,a) \coloneqq a^{q-2}t^2/2$ for all $a \geq 0 $ and $t\in \mathbb{R}$. The function satisfies
\begin{align*}
(\nabla^2\beta(t,a)) = \begin{pmatrix}
a^{q-2}&(q-2)a^{q-3}t\\
(q-2)a^{q-3}t & (q-2)(q-3)a^{q-4}t^2/2
\end{pmatrix}.
\end{align*}
Since $0\leq a^{q-2}$ and $0\leq \det((\nabla^2\beta)(t,a) = a^{2q-6}t^2(2-q)(q-1)$ for all $a \geq 0$ and $t\in \mathbb{R}$, this matrix is non-negative definite for all $1<q<2$ and so the relaxed energy $\mathcal{J}^*$ is convex in $\tau$ and $a$.
\begin{remark}[Opposite case]\label{rem:OppositeCase}
The restriction $1<q<2$ is equivalent to $2<p<\infty$. Therefore, this approach covers the range of $p$ that is excluded in \cite{DieningFornasierWank17}.
\end{remark}
Given $a\colon \Omega \to [0,\infty)$, the Euler-Lagrange equation corresponding to the minimization of \eqref{eq:relaxedFunctional} over $\Wdivf$ seeks $\sigma \in \Wdiv$ and $u \in L^p(\Omega)$ with 
\begin{align*}
a^{q-2} \sigma - \nabla u = 0\qquad\text{and} \qquad -\textup{div}\,\sigma = f.
\end{align*}
This saddle point system degenerates as $\essinf a \to 0$ and $\esssup a \to \infty$. We overcome this difficulty by restricting the minimization of \eqref{eq:relaxedFunctional} for fixed $\tau\in \Wdiv$ to functions $a:\Omega \to [\varepsilon_-,\varepsilon_+]$ within a relaxation interval $\varepsilon = [\varepsilon_-,\varepsilon_+] \subset (0,\infty)$. Differentiation shows that the minimizer for a fixed $\tau \in \Wdiv$ reads 
\begin{align}\label{eq:minimalA}
\argmin_{a\colon \Omega \to [\varepsilon_-,\varepsilon_+]} \mathcal{J}^*(\tau,a) = \varepsilon_- \vee |\tau| \wedge \varepsilon_+.
\end{align}
This leads for all $\tau \in \Wdiv$ to the relaxed energy
\begin{align*}
\mathcal{J}^*_\varepsilon(\tau) \coloneqq \mathcal{J}^*(\tau,\varepsilon_- \vee |\tau| \wedge \varepsilon_+) = \min_{a:\Omega\to[\varepsilon_-,\varepsilon_+]} \mathcal{J}^*(\tau,a).
\end{align*}
This monotonically decreasing functional (with respect to growing intervals $\varepsilon = [\varepsilon_-,\varepsilon_+]$) hides the minimization with respect to $a:\Omega \to [\varepsilon_-,\varepsilon_+]$. The functional can be written in terms of the integrand
\begin{align}\label{eq:defKappa}
\begin{aligned}
\kappa^*_\varepsilon(t) \coloneqq \begin{cases}
\frac{1}{2}\varepsilon_-^{q-2}t^2 + \left(\frac{1}{q}-\frac{1}{2}\right)\varepsilon_-^q&\text{for }t\leq \varepsilon_-,\\
\frac{1}{q}t^q&\text{for } \varepsilon_-\leq t\leq \varepsilon_+,\\
\frac{1}{2}\varepsilon_+^{q-2}t^2+\left(\frac{1}{q}-\frac{1}{2}\right)\varepsilon_+^q&\text{for }\varepsilon_+ \leq t.
\end{cases}
\end{aligned}
\end{align}
More precisely, the functional equals
\begin{align*}
\mathcal{J}_\varepsilon^*(\tau) = \int_\Omega \kappa^*_\varepsilon(|\tau|)\,\mathrm{d}x\qquad\text{for all }\tau \in \Wdiv.
\end{align*}
The function $\kappa^*_\varepsilon$ has quadratic growth in the sense that $\kappa^*_\varepsilon(t) \simeq \varepsilon_+^{q-2} t^2$  for $t \geq \epsilon_+$. Therefore, the relaxed energy $\mathcal{J}^*_\varepsilon(v) < \infty$ with $\epsilon_+<\infty$ is bounded if and only if $\tau \in \WdivTwo$. 
This shows that there exists a function in $\Wdivf$ with finite energy if and only if the right-hand side $f$ allows for the existence of a function $\tau \in \WdivTwo$ with $-\textup{div}\,\tau = f$.
Since the minimizer $\sigma$ in  \eqref{eq:dualProblemIntro} satisfies $-\textup{div}\, \sigma = f$, sufficient conditions are the following.
\begin{theorem}[Maximal regularity]\label{thm:maxRegularity} Let $\sigma$ be the minimizer in \eqref{eq:dualProblemIntro} and  let $\Omega \subset \RRd$ be a bounded domain.
\begin{enumerate}
\item If $\Omega$ has a $C^{1,\alpha}$-boundary $\partial \Omega$ for some $\alpha \in (0,1]$, we have for any $r\geq q$ \label{eq:itm1asd}
\begin{align*}
\lVert \sigma \rVert_{L^r(\Omega)} \lesssim \lVert f \rVert_{W^{-1,r}(\Omega)}.
\end{align*}
\item If $\Omega$ is a convex open set with $d\geq 2$, we have\label{eq:itm2asd}
\begin{align*}
\lVert \sigma \rVert_{L^\frac{2d}{d-2}(\Omega)} \lesssim \lVert \nabla \sigma \rVert_{L^2(\Omega)} \lesssim \lVert f \rVert_{L^2(\Omega)}.
\end{align*}
\end{enumerate}
\end{theorem}
\begin{proof}
  The statement in \ref{eq:itm1asd} is shown for equations in \cite[Thm.~1.6]{KinnunenZhou01} and for systems in \cite[Thm.~4.1]{BreCiaDieKuuSch18}.
  The first estimate in \ref{eq:itm2asd} follows from the Sobolev embedding theorem, the second is shown in \cite[Thm.~2.3]{CianchiMazya18} for equations and in \cite{CiaMaz19systems,BalciCianchiDieningMazya22} for systems.
\end{proof}
If the set $\WdivTwof$ is not empty, the direct method in the calculus of variations leads to the existence of a unique minimizer 
\begin{align}\label{eq:regMinimizer}
\sigma_\varepsilon = \argmin_{\tau \in \Wdivf} \mathcal{J}_\varepsilon^*(\tau).
\end{align}
\section{Convergence in the relaxation parameter}\label{sec:conInRelaxPara}
This section shows that the minimizer $\sigma_\varepsilon\in \WdivTwof$ of the relaxed energy $\mathcal{J}^*_\varepsilon$ converges to the minimizer $\sigma
\in \WdivTwof$ of $\mathcal{J}^*$ as the interval $\varepsilon \to (0,\infty)$. In particular, we derive an upper bound for the relaxation error $\sigma_\varepsilon - \sigma$.
\begin{theorem}[Convergence in $\varepsilon$]\label{thm:ConvInRelaxPara}
Suppose that the minimizers $\sigma_\varepsilon \in \Wdivf$ with $\varepsilon = [\varepsilon_-,\varepsilon_+] \subset (0,\infty)$ exist and that $\sigma \in L^r(\Omega;\mathbb{R}^d)$ for some $r\geq 2$. 
Then the energy difference is bounded by
\begin{align}\label{eq:ConvInRelaxPara}
\begin{aligned}
\mathcal{J}^*(\sigma_\varepsilon) - \mathcal{J}^*(\sigma)&\leq \mathcal{J}_\varepsilon^*(\sigma_\varepsilon) - \mathcal{J}^*(\sigma) \leq  \frac{|\Omega|}{q}\varepsilon_-^q + \frac{1}{q}\varepsilon_+^{-(r-q)} \lVert \sigma \rVert^r_{L^r(\Omega)}.
\end{aligned}
\end{align}
\end{theorem}
\begin{proof}
The first inequality in \eqref{eq:ConvInRelaxPara} follows by the monotonicity of $\mathcal{J}_\varepsilon$ with respect to the relaxation parameter.
The minimizing property of $\sigma_\varepsilon \in \WdivTwof$ implies
\begin{align}\label{eq:ProofConv1}
\begin{aligned}
\mathcal{J}_\varepsilon^*(\sigma_\varepsilon) - \mathcal{J}^*(\sigma) & \leq \mathcal{J}_\varepsilon^*(\sigma) - \mathcal{J}^*(\sigma) = \int_\Omega \left(\kappa_\varepsilon(|\sigma|) - \frac{1}{q} |\sigma|^q\right)\mathrm{d}x.
\end{aligned}
\end{align}
It follows from the definition of $\kappa_\varepsilon$ in \eqref{eq:defKappa} that
\begin{align}\label{eq:ProofConv2}
\kappa_\varepsilon(|\sigma|) - \frac{1}{q}|\sigma|^q \leq \begin{cases}
\varepsilon_-^q/q&\text{in }\lbrace |\sigma| \leq \varepsilon_-\rbrace,\\
0&\text{in } \lbrace \varepsilon_-< |\sigma|\leq \varepsilon_+\rbrace,\\
\varepsilon_+^{q-2} |\sigma|^2/q&  \text{in } \lbrace \varepsilon_+ <  |\sigma| \rbrace.
\end{cases}
\end{align}
The inequalities in \eqref{eq:ProofConv1}--\eqref{eq:ProofConv2} yield
\begin{align}\label{eq:ProofConv3}
\begin{aligned}
\mathcal{J}^*_\varepsilon(\sigma_\varepsilon) - \mathcal{J}^*(\sigma)&\leq \frac{|\Omega|}{q}\varepsilon_-^q + \frac{\varepsilon_+^{q-2}}{q} \int_{\lbrace \varepsilon_+ < |\sigma|\rbrace} |\sigma|^2\,\mathrm{d}x.
\end{aligned}
\end{align}
H\"older's inequality bounds the second integral by
\begin{align}\label{eq:ProofConv4}
\begin{aligned}
&\int_{\lbrace \varepsilon_+ < |\sigma| \rbrace} |\sigma|^2\,\mathrm{d}x = \int_\Omega \mathds{1}_{\lbrace \varepsilon_+ < |\sigma| \rbrace} |\sigma|^2\,\mathrm{d}x\\
&\qquad \leq \varepsilon_+^{2-r} \lVert \varepsilon_+^{r-2} \mathds{1}_{\lbrace \varepsilon_+ < |\sigma| \rbrace} \rVert_{L^{r/(r-2)}(\Omega)} \lVert \sigma \rVert_{L^r(\Omega)}^2 \leq \varepsilon_+^{2-r} \lVert \sigma \rVert_{L^r(\Omega)}^r.
\end{aligned}
\end{align}
Combining \eqref{eq:ProofConv3}--\eqref{eq:ProofConv4} concludes the proof.
\end{proof}
Convergence of the energies yields convergence with respect to a natural distance. More precisely, we have the following. 
\begin{definition}[Quantities]\label{def:Quantites}
Let $P \in \mathbb{R}^d$ and $\varepsilon = [\varepsilon_-,\varepsilon_+] \subset [0,\infty]$. Recall that 
$(\phi^*_\varepsilon)'(|P|) = (\varepsilon_- \vee |P| \wedge \varepsilon_+)^{q-2}|P|$
and set 
\begin{align*}
A^*_\varepsilon(P) \coloneqq \begin{cases}
\frac{(\phi_\varepsilon^*)'(|P|)}{|P|}P&\text{for }P\neq 0,\\
0&\text{for }P = 0
\end{cases}\quad\text{and}\quad
V^*_\varepsilon(P) \coloneqq \begin{cases}
\sqrt{\frac{(\phi_\varepsilon^*)'(|P|)}{|P|}}P&\text{for }P\neq 0,\\
0&\text{for }P = 0.
\end{cases}
\end{align*} 
If $\varepsilon = [0,\infty]$, we write $A^* \coloneqq A^*_{[0,\infty]}$ and $V^* \coloneqq V^*_{[0,\infty]}$.
\end{definition}
\begin{lemma}[Equivalence]\label{lem:equivalence}
For all $P,Q\in \mathbb{R}^d$ holds the equivalence
\begin{align*}
(A^*_\varepsilon(P)-A^*_\varepsilon(Q))\cdot (P-Q) \simeq \frac{(\phi_\varepsilon^*)'(|P| \vee |Q|)}{|P| \vee |Q|} |P-Q|^2 \simeq |V^*_\varepsilon(P)-V^*_\varepsilon(Q)|^2.
\end{align*} 
The hidden constants depend solely on $q$ and are in particular independent of $\varepsilon$.
\end{lemma}
\begin{proof}
This equivalence follows from the uniform convexity of $\phi_\varepsilon^*$, see \cite[Lem.\ 2]{DieningFornasierWank17} and \cite[Lem.~3]{DieningEttwein08} for more details.
\end{proof}
\begin{lemma}[Natural distance]\label{lem:Quasi-Norm}
The minimizer $\sigma_\varepsilon$ of $\mathcal{J}^*_\varepsilon$ over $\Wdivf$ satisfies, for all $\tau \in \Wdivf$ and $\varepsilon \subset [0,\infty]$,
\begin{align*}
\mathcal{J}_\varepsilon^*(\tau) - \mathcal{J}_\varepsilon^*(\sigma_\varepsilon) &\leq \int_\Omega (A^*_\varepsilon(\tau)-A^*_\varepsilon(\sigma_\varepsilon))\cdot (\tau - \sigma_\varepsilon)\,\mathrm{d}x\\
&\lesssim \int_\Omega |V^*_\varepsilon(\tau) - V^*_\varepsilon(\sigma_\varepsilon)|^2\,\mathrm{d}x\lesssim \mathcal{J}_\varepsilon^*(\tau) - \mathcal{J}_\varepsilon^*(\sigma_\varepsilon).
\end{align*}
The hidden constants are independent of $\varepsilon$. The equivalence for $\varepsilon = [0,\infty]$ shows 
\begin{align}\label{eq:QuasiNorm}
\mathcal{J}^*(\tau) - \mathcal{J}^*(\sigma) \simeq \int_\Omega |V^*(\tau) - V^*(\sigma)|^2\,\mathrm{d}x\qquad\text{for all }\tau \in \Wdivf.
\end{align}
\end{lemma}
\begin{proof}
This lemma is proven in \cite[Lem.17]{DieningFornasierWank17}.
\end{proof}
The equivalence in \eqref{eq:QuasiNorm} shows that the convergence of the minimizers $\sigma_\varepsilon \to \sigma$ with respect to the natural distance $\int_\Omega |V^*(\sigma) - V^*(\sigma_\varepsilon)|^2\,\mathrm{d}x$ is equivalent to the convergence of the energies $\mathcal{J}^*(\sigma_\varepsilon) \searrow \mathcal{J}^*(\sigma)$.
Therefore, the convergence result in Theorem~\ref{thm:ConvInRelaxPara} implies convergence with respect to the natural distance. 
\section{Ka\v{c}anov iterations}\label{sec: Kacanov}
This section investigates the convergence of the following fixed-point iteration with fixed interval $\varepsilon = [\varepsilon_-,\varepsilon_+]\subset (0,\infty)$. Given $\tau_0 \in \Wdivf$, we calculate iteratively the functions $\tau_{n+1}\in \Wdivf$ and $u_{n+1} \in L^p(\Omega)$ by solving the saddle point problem, for all $\xi \in\Wdiv$ and $v\in L^p(\Omega)$,
\begin{align}\label{eq:SaddlePointKacanov}
\begin{aligned}
\int_\Omega (\varepsilon_- \vee |\tau_n|\wedge \varepsilon_+)^{q-2} \tau_{n+1} \cdot \xi \,\mathrm{d}x + \int_\Omega u_{n+1}\, \textup{div}\, \xi\,\mathrm{d}x & = 0,\\
\int_\Omega v\, \textup{div}\, \tau_{n+1}\,\mathrm{d}x & = -\int_\Omega f v\,\mathrm{d}x.
\end{aligned}
\end{align}
The following theorem shows that the resulting sequence $(\tau_n)_{n\in \mathbb{N}_0}\subset \Wdivf$ converges towards the minimizer $\sigma_\varepsilon \in \Wdivf$ defined in \eqref{eq:regMinimizer}.
\begin{theorem}[Exponential decay]\label{thm:ExpDecay}
There exists a constant $C_q < \infty$ depending solely on $q$ such that the iterates $\tau_n$ from \eqref{eq:SaddlePointKacanov} satisfy, with $\delta = C_q^{-1}(\varepsilon_-/\varepsilon_+)^{2-q}$,
\begin{align}\label{eq:expDecay}
\delta \big(\mathcal{J}^*_\varepsilon (\tau_n) -  \mathcal{J}^*_\varepsilon(\sigma_\varepsilon)\big) \leq   \mathcal{J}_\varepsilon^*(\tau_n) - \mathcal{J}^*(\tau_{n+1})\qquad\text{for all }n\in \mathbb{N}_0.
\end{align}
\end{theorem}
Before we prove this theorem, let us emphasize its consequences.
\begin{corollary}[Convergence]
Recall the constant $\delta$ from Theorem \ref{thm:ExpDecay}. It holds
\begin{align*}
\mathcal{J}^*_\varepsilon(\tau_{n+1}) - \mathcal{J}^*_\varepsilon(\sigma_\varepsilon) \leq (1-\delta)^n \big( \mathcal{J}_\varepsilon^*(\tau_0) - \mathcal{J}^*_\varepsilon(\sigma_\varepsilon)\big)\qquad\text{for all }n\in\mathbb{N}.
\end{align*}
\end{corollary}
\begin{proof}
Due to \eqref{eq:expDecay} we have for all $n\in \mathbb{N}$
\begin{align*}
\mathcal{J}^*_\varepsilon(\tau_{n+1}) - \mathcal{J}^*_\varepsilon(\sigma_\varepsilon) & = \big(\mathcal{J}^*_\varepsilon(\tau_n) - \mathcal{J}^*_\varepsilon(\sigma_\varepsilon)\big) - \big(\mathcal{J}^*_\varepsilon(\tau_{n}) - \mathcal{J}^*_\varepsilon(\tau_{n+1})\big)\\
& \leq (1-\delta) \big( \mathcal{J}^*_\varepsilon(\tau_{n}) - \mathcal{J}^*_\varepsilon(\sigma_\varepsilon)\big). 
\end{align*}
In particular, the energy error reduces in each iteration by the factor $(1-\delta)$, leading to the corollary.
\end{proof}
\begin{proof}[Proof of Theorem \ref{thm:ExpDecay}]
Recall the mapping $A^*_\varepsilon$ from Definition \ref{def:Quantites}. Lemma \ref{lem:Quasi-Norm} and the definition of $\tau_n$ in \eqref{eq:SaddlePointKacanov} imply for all positive values $\gamma>0$ that
\begin{align}\label{eq:addends}
\begin{aligned}
&\mathcal{J}^*_\varepsilon(\tau_n) - \mathcal{J}^*_\varepsilon(\sigma_\varepsilon) \leq \int_\Omega \big(A^*_\varepsilon(\tau_n) - A^*_\varepsilon(\sigma_\varepsilon)\big) \cdot (\tau_n-\sigma_\varepsilon)\,\mathrm{d}x\\
&\qquad  = \int_\Omega \frac{(\varphi_\varepsilon^*)'(|\tau_n|)}{|\tau_n|} (\tau_n - \tau_{n+1})\cdot (\tau_n - \sigma_\varepsilon)\,\mathrm{d}x\\
&\qquad \leq \frac{1}{2\gamma} \int_\Omega \frac{(\phi_\varepsilon^*)'(|\tau_n|)}{|\tau_n|} |\tau_n-\tau_{n+1}|^2\,\mathrm{d}x + \frac{\gamma}{2} \int_\Omega \frac{(\phi_\varepsilon^*)'(|\tau_n|)}{|\tau_n|} |\tau_n-\sigma_\varepsilon|^2\,\mathrm{d}x.
\end{aligned}
\end{align}
We define $\mathcal{J}_\varepsilon^*(\tau,a)\coloneqq \mathcal{J}^*(\tau,\varepsilon_-\vee a \wedge \varepsilon_+)$ with the relaxed functional $\mathcal{J}^*(\bigcdot,\bigcdot)$ from~\eqref{eq:relaxedFunctional}. Then \eqref{eq:SaddlePointKacanov} (combined with $\textup{div}\, (\tau_n - \tau_{n+1}) = 0$) and \eqref{eq:minimalA} show for the first addend in \eqref{eq:addends} that
\begin{align*}
&\frac{1}{2\gamma} \int_\Omega \frac{(\phi_\varepsilon^*)'(|\tau_n|)}{|\tau_n|} |\tau_n-\tau_{n+1}|^2\,\mathrm{d}x  \\
&\quad =\frac{1}{2\gamma} \int_\Omega \frac{(\phi_\varepsilon^*)'(|\tau_n|)}{|\tau_n|} \tau_n\cdot \tau_n\,\mathrm{d}x - \frac{1}{2\gamma} \int_\Omega \frac{(\phi_\varepsilon^*)'(|\tau_n|)}{|\tau_n|} \tau_n\cdot  \tau_{n+1} \,\mathrm{d}x \\
&\quad =\frac{1}{2\gamma} \int_\Omega \frac{(\phi_\varepsilon^*)'(|\tau_n|)}{|\tau_n|} |\tau_n|^2\,\mathrm{d}x - \frac{1}{2\gamma} \int_\Omega \frac{(\phi_\varepsilon^*)'(|\tau_n|)}{|\tau_n|} |\tau_{n+1}|^2 \,\mathrm{d}x \\
&\quad = \gamma^{-1}\big(\mathcal{J}^*_\varepsilon(\tau_n,|\tau_n|) - \mathcal{J}^*_\varepsilon(\tau_{n+1},|\tau_n|)\big)\leq \gamma^{-1}\big(\mathcal{J}^*_\varepsilon(\tau_n,|\tau_n|) - \mathcal{J}^*_\varepsilon(\tau_{n+1},|\tau_{n+1}|)\big)\\
&\quad = \gamma^{-1}\big(\mathcal{J}_\varepsilon^*(\tau_n) - \mathcal{J}_\varepsilon^*(\tau_{n+1})\big).
\end{align*}
Since $\varepsilon_+^{q-2} \leq (\phi_\varepsilon^*)'(t)/t = (\varepsilon_-\vee t \wedge \varepsilon_+)^{q-2} \leq \varepsilon_-^{q-2}$, it holds 
\begin{align*}
\frac{(\phi_\varepsilon^*)'(t)}{t} \leq \left(\frac{\varepsilon_+}{\varepsilon_-}\right)^{2-q} \frac{(\phi_\varepsilon^*)'(s)}{s}\qquad\text{for all }t,s>0.
\end{align*}
This inequality and Lemma \ref{lem:equivalence}--\ref{lem:Quasi-Norm} show that the second addend in \eqref{eq:addends} satisfies, with some constant $C_q<\infty$ that depends solely on $q$,  
\begin{align*}
\frac{\gamma}{2} \int_\Omega \frac{(\phi_\varepsilon^*)'(|\tau_n|)}{|\tau_n|} |\tau_n-\sigma_\varepsilon|^2\,\mathrm{d}x & \leq \frac{\gamma}{2}\left(\frac{\varepsilon_+}{\varepsilon_-}\right)^{2-q} \int_\Omega \frac{(\phi_\varepsilon^*)'(|\tau_n| \vee |\sigma_\varepsilon|)}{|\tau_n|\vee |\sigma_\varepsilon|} |\tau_n - \sigma_\varepsilon|^2\,\mathrm{d}x\\
&\leq \gamma C_q \left(\frac{\varepsilon_+}{\varepsilon_-}\right)^{2-q} \big(\mathcal{J}_\varepsilon^*(\tau_n) - \mathcal{J}_\varepsilon^*(\sigma_\varepsilon)\big).
\end{align*}
The estimate for the two addends in \eqref{eq:addends} yield for all $\gamma >0$ that
\begin{align*}
\gamma(1-\gamma C_q(\varepsilon_+/\varepsilon_-)^{2-q})\big(\mathcal{J}_\varepsilon^*(\tau_n) - \mathcal{J}_\varepsilon^*(\sigma_\varepsilon)\big) \leq \mathcal{J}_\varepsilon(\tau_n) - \mathcal{J}_\varepsilon(\tau_{n+1}).
\end{align*}
This inequality and $\max_{\gamma>0} \gamma(1-C_q\gamma(\varepsilon_+/\varepsilon_-)^{2-q}) = (4C_q)^{-1}(\varepsilon_-/\varepsilon_+)^{2-q}$ conclude the proof.
\end{proof}
\section{Algebraic rate}\label{sec:AlgebraicRate}
Since the rate of convergence of the \Kacanov\ scheme depends strongly on the choice of the interval $\varepsilon$, we have to balance the increase of the intervals $\varepsilon$ and the decreasing speed of convergence carefully. This section shows that it is possible to get some algebraic rate of convergence.

Let $\alpha,\beta>0$ with $\alpha + \beta \leq (2-q)^{-1}$ and define the increasing sequence of intervals $\varepsilon_n = (\varepsilon_{n,-},\varepsilon_{n,+}) = ((n+1)^{-\alpha},(n+1)^\beta)$ for all $n\in \mathbb{N}$. Let $\sigma_{n+1}\in \Wdivf$ denote the solution to \eqref{eq:SaddlePointKacanov} with interval $\varepsilon \coloneqq \varepsilon_n$. Theorem \ref{thm:ExpDecay} shows that the iterates $\sigma_n$ satisfy the decay estimate, for all $n\in \mathbb{N}$ and $\delta_n \coloneqq C^{-1}(\varepsilon_{n,-}/\varepsilon_{n,+})^{2-q}$,
\begin{align}\label{eq:InequAlgebrRate1}
\delta_n\, \big(\mathcal{J}^*_{\varepsilon_n}(\sigma_n) - \mathcal{J}^*_{\varepsilon_n}(\sigma_{\varepsilon_n})\big) \leq \mathcal{J}^*_{\varepsilon_n}(\sigma_{n+1}) - \mathcal{J}^*_{\varepsilon_n}(\sigma_n).
\end{align} 
Since $\mathcal{J}_{\varepsilon_{n+1}}^* \leq \mathcal{J}_{\varepsilon_n}^*$, the inequality in \eqref{eq:InequAlgebrRate1} yields for all $n\in\mathbb{N}$ that
\begin{align}\label{eq:Est1}
\begin{aligned}
&\mathcal{J}_{\varepsilon_{n+1}}^* (\sigma_{n+1})- \mathcal{J}^*(\sigma)\leq \mathcal{J}_{\varepsilon_{n}}^* (\sigma_{n+1}) - \mathcal{J}^*(\sigma)\\
&\qquad = \big(\mathcal{J}_{\varepsilon_{n}}^* (\sigma_{n}) - \mathcal{J}^*(\sigma)\big) - \big(\mathcal{J}_{\varepsilon_{n}}^* (\sigma_{n}) - \mathcal{J}_{\varepsilon_n}^* (\sigma_{n+1})\big) \\
&\qquad\leq \big(\mathcal{J}_{\varepsilon_{n}}^* (\sigma_{n}) - \mathcal{J}^*(\sigma)\big) - \delta_n\, \big(\mathcal{J}_{\varepsilon_{n}}^* (\sigma_{n}) - \mathcal{J}_{\varepsilon_n}^* (\sigma_{\varepsilon_n})\big)\\
&\qquad  =(1-\delta_n)\,\big(\mathcal{J}_{\varepsilon_{n}}^* (\sigma_{n}) - \mathcal{J}^*(\sigma)\big) + \delta_n\, \big(\mathcal{J}_{\varepsilon_n}^* (\sigma_{\varepsilon_n})-\mathcal{J}^*(\sigma)\big).
\end{aligned}
\end{align}  
The combination of the estimates in \eqref{eq:Est1} and Theorem~\ref{thm:ConvInRelaxPara} results with some fixed constant $c_R < \infty$ and $r\geq 2$ such that $\lVert \sigma \rVert_{L^r(\Omega)} <\infty$ in 
\begin{align*}
\mathcal{J}_{\varepsilon_{n+1}}^*(\sigma_{n+1}) - \mathcal{J}^*(\sigma)\leq (1-\delta_n) \big(\mathcal{J}_{\varepsilon_{n}}^* (\sigma_{n}) - \mathcal{J}^*(\sigma)\big) + \delta_n c_R \,(\varepsilon_-^q + \varepsilon_+^{-(r-q)}).
\end{align*}
The last term is very small for large $n$. However, the factor $(1-\delta_n) \to 1$ as $n \to \infty$. Nevertheless, the slow convergence of $\delta_n\to 0$ allows for an algebraic convergence of the product $\Pi_{i=1}^n(1-\delta_n)$ towards zero. This observation is key in the proof of the algebraic rate of convergence~\cite[Sec.\ 5]{DieningFornasierWank17}. The proof extends without modifications to the \Kacanov\ scheme in this paper and leads to the following statement.
\begin{theorem}[Algebraic rate]\label{thm:AlgebraicRate}
Suppose the minimizer $\sigma$ in \eqref{eq:dualProblemIntro} satisfies $\lVert \sigma \rVert_{L^r(\Omega)} < \infty$ for some $r>2$. Then there exist constants $0<c$ and $C<\infty$ that depends on $\alpha,\beta,r,q, |\Omega|, \lVert \sigma \rVert_{L^r(\Omega)}$, and the initial error $\mathcal{J}^*_{\varepsilon_0}(\sigma_0)- \mathcal{J}^*(\sigma)$ such that 
\begin{align*}
\mathcal{J}^*_{\varepsilon_n}(\sigma_n) - \mathcal{J}^*(\sigma) \leq C n^{-1/c}.
\end{align*}
\end{theorem}
\section{Numerical realization}\label{sec:NumerRealization}
In this section we introduce a discretization and suggest an adaptive scheme to improve the convergence of the iterative method.
\subsection{Discretization}\label{sec:Discretization}
Conforming discretizations of the regularized linear dual problem in \eqref{eq:SaddlePointKacanov} lead to saddle point formulations which are more challenging to solve than (symmetric positive definite) primal formulations. 
One possible remedy is the use of duality relations to reformulate dual formulations with lowest-order Raviart-Thomas elements as primal problems using Crouzeix-Raviart elements (see for example~\cite{Marini85,CarstensenLiu15,Bartels21,BartelsKaltenbach22}).
An alternative is the use of a non-conforming discretization of the dual problem. We set the lowest-order Lagrange finite element space
\begin{align*}
S^1_{0}(\tria) \coloneqq \lbrace v_h \in W^{1,q}_0(\Omega) \colon v_h|_T \in \mathbb{P}_1(T)\text{ for all }T\in \tria\rbrace.
\end{align*}
The non-conforming discretization of the dual problem results as follows.
Let $u_h \in S^1_0(\tria)$ denote the minimizer of the energy in \eqref{eq:PrimalProblemIntro} over the discrete subspace $S^1_0(\tria)$.
Due to the Euler-Lagrange equations the minimizer is characterized as the unique solution to the problem 
\begin{align}\label{eq:PrimalDiscrete}
\int_\Omega |\nabla u_h|^{p-2} \nabla u_h \cdot \nabla v_h \dx = \int_\Omega f v_h\dx\qquad\text{for all }v_h \in S^1_{0}(\tria).
\end{align}
We set the discrete divergence $\textup{div}_h\colon \mathbb{P}_0(\tria;\mathbb{R}^d) \to S_0^1(\tria)^*$ such that
\begin{align*}
  \skp{\textup{div}_h \tau_h}{v_h}_{S^1_0(\tria)^*, S^1_0(\tria)} \coloneqq - \int_\Omega \nabla v_h\cdot \tau_h\dx \quad\text{for all }\tau_h\in \mathbb{P}_0(\tria;\mathbb{R}^d)\text{ and }v_h \in S_0^1(\tria).
\end{align*}
We interpret $\textup{div}_h \tau_h \in S_0^1(\tria)^*$ by its Riesz representative $\textup{div}_h \tau_h \in S_0^1(\tria)$, i.e.
\begin{align*}
  \int_\Omega v_h \, \textup{div}_h \tau_h \dx = \skp{\textup{div}_h \tau_h}{v_h}_{S^1_0(\tria)^*, S^1_0(\tria)}  \quad\text{for all }\tau_h\in \mathbb{P}_0(\tria;\mathbb{R}^d)\text{ and }v_h \in S_0^1(\tria).
\end{align*}
\begin{lemma}[Dual formulation]\label{lem:dualForm}
  Let $u_h\in S_0^1(\tria)$ solve \eqref{eq:PrimalDiscrete} and define the piece-wise constant function $\sigma_h \coloneqq |\nabla u_h|^{p-2} \nabla u_h\in \mathbb{P}_0(\tria;\mathbb{R}^d)$.  Then we have
  \begin{align*}
    \begin{aligned}
      \int_\Omega |\sigma_h|^{q-2}\sigma_h \cdot \tau_h \dx + \int_\Omega u_h \,\textup{div}_h \tau_h \dx & = 0&&\text{for all }\tau_h \in \mathbb{P}_0(\tria;\mathbb{R}^d),
      \\
      \int_\Omega v_h \,\textup{div}_h \sigma_h \dx & = -\int_\Omega f v_h\dx&&\text{for all }v_h \in S^1_{0}(\tria).
    \end{aligned}
  \end{align*}
\end{lemma}
The lemma follows directly from abstract duality theory. A more direct proof reads as follows.
\begin{proof}[Proof of Lemma~\ref{lem:dualForm}]
  Let the functions $u_h$ and $\sigma_h$ be as in the lemma. Using the definition of the discrete divergence shows that the function $\sigma_h$ fulfills the second equation in the lemma. The first equation is equivalent to 
  \begin{align}\label{eq:ProofequaiTemp}
    |\sigma_h|^{q-2}\sigma_h = \nabla u_h.
  \end{align}
  This yields $|\sigma_h|^{q-1} = |\nabla u_h|$, which is due to the identity $1/p+1/q = 1$ equivalent to $|\sigma_h|^{q-2} = |\nabla u_h|^{2-p}$. Hence, \eqref{eq:ProofequaiTemp} is equivalent to $\sigma_h = |\nabla u_h|^{p-2} \nabla u_h$.
\end{proof}
The mixed problem in Lemma~\ref{lem:dualForm} corresponds to the minimization of $\mathcal{J}^*$ over
\begin{align*}
&\mathbb{P}_0(\textup{div}_h{=}{\shortminus}f,\tria) \coloneqq \lbrace \tau_h \in \mathbb{P}_0(\tria,\mathbb{R}^d) \colon \textup{div}_h \tau_h = - f\text{ in } S_0^1(\tria)^*\rbrace.
\end{align*}
More precisely, we have
\begin{align}\label{eq:discDivMinProb}
|\nabla u_h|^{p-2} \nabla u_h = \sigma_h = \min_{\tau_h \in \mathbb{P}_0(\textup{div}_h{=}{\shortminus}f,\tria)} \mathcal{J}^*(\tau_h).
\end{align}
The discrete version of the \Kacanov{} scheme in \eqref{eq:SaddlePointKacanov} corresponding to the minimization problem in \eqref{eq:discDivMinProb} reads as follows. 
Given an interval $\varepsilon \subset (0,\infty)$ and some initial data $\sigma_{h,0} \in \mathbb{P}_0(\tria;\mathbb{R}^d)$, we compute iteratively for all $n\in \mathbb{N}$ the solution $\sigma_{h,n} \in \mathbb{P}_0(\tria;\mathbb{R}^d)$ and $u_{h,n} \in S_0^1(\tria)$ with, for all $\xi \in \mathbb{P}_0(\tria;\mathbb{R}^d)$ and $v_h\in S_0^1(\tria)$,  
\begin{align}\label{eq:discreteDualProblem}
\begin{aligned}
\int_\Omega (\varepsilon_- \vee |\sigma_{h,n}|\wedge \varepsilon_+)^{q-2} \sigma_{h,n+1} \cdot \xi_h \,\mathrm{d}x -  \int_\Omega u_{h,n+1}\, \textup{div}_h\, \xi_h\,\mathrm{d}x & = 0,\\
\int_\Omega v_h\, \textup{div}_h\, \sigma_{h,n+1}\,\mathrm{d}x & = -\int_\Omega f v_h\,\mathrm{d}x.
\end{aligned}
\end{align}
\begin{lemma}[Primal formulation]
Given $\sigma_{h,n} \in \mathbb{P}_0(\tria;\mathbb{R}^d)$, let $u_{h,{n+1}}\in S^1_0(\tria)$ be the unique solution to
\begin{align}\label{eq:PrimalKacanov}
\int_\Omega (\varepsilon_- \vee |\sigma_{h,n}|\wedge \varepsilon_+)^{2-q} \nabla u_{h,n+1}\cdot \nabla v_h\dx = \int_\Omega f v_h\dx\quad\text{for all }v_h \in S_0^1(\tria).
\end{align}
Then $\sigma_{h,n+1} \coloneqq (\varepsilon_- \vee |\sigma_{h,n}|\wedge \varepsilon_+)^{2-q} \nabla u_{h,n+1}$ and $u_{h,n+1}$ solve \eqref{eq:discreteDualProblem}.
\end{lemma}
\begin{proof}
Similar calculations as in the proof of Lemma~\ref{lem:dualForm} yield the result.
\end{proof}
A review of the proofs in Section~\ref{sec:conInRelaxPara}--\ref{sec:AlgebraicRate} shows that most results remain valid without any changes in the proofs if we replace  $\Wdivf$ by $\mathbb{P}_0(\textup{div}_h{=}{\shortminus}f,\tria)$.
The only exception are the maximal regularity results in Theorem~\ref{thm:maxRegularity} which, to our knowledge, do not exist for discrete approximations of the $p$-Laplacian and motivates further research.

\subsection{Adaptive scheme}\label{sec:adaptScheme}
To improve the algebraic rate of convergence from Theorem \ref{thm:AlgebraicRate}, we suggest an adaptive scheme that takes into account the errors caused by the discretization, the regularization, and the \Kacanov{} iteration. Even so there is some progress in the numerical analysis of such schemes (see for example \cite{CongreveWihler17,DiPietroVohralikYousef15,GantnerHaberlPraetoriusStiftner18}), a rigorous analysis of the convergence of the proposed scheme is beyond the scope of this paper.
We indicate beneficial properties of this adaptive approach by numerical experiments in Section~\ref{sec:NumExp}. 

We set the functions  $\varphi_\varepsilon'(t) \coloneqq ((\phi_\varepsilon^*)')^{-1}(t)$ and $\kappa_\varepsilon(t) \coloneqq (\kappa_\varepsilon^*)^*(t) \coloneqq \sup \lbrace r\geq 0 \colon rt - \kappa^*_\varepsilon(r)\rbrace$ for all $t\geq 0$. The latter contribution equals
\begin{align*}
\begin{aligned}
\kappa_\varepsilon(t)& = \int_0^t \varphi_\varepsilon'(s)\,\mathrm{d}s  - \kappa^*_\varepsilon(0)  = 
\begin{cases}
\frac{1}{2}\varepsilon_-^{2-q}t^2 - \left(\frac{1}{q}-\frac{1}{2}\right) \varepsilon_-^q&\text{for }t^p\leq \varepsilon_-^q,\\
\frac{1}{p}t^p &\text{for }\varepsilon_-^q \leq t^p \leq \varepsilon_+^q,\\
\frac{1}{2}\varepsilon_+^{2-q}t^2  - \left(\frac{1}{q}-\frac{1}{2}\right) \varepsilon_+^q&\text{for }\varepsilon_+^q\leq t^p. 
\end{cases}
\end{aligned}
\end{align*}
It defines the primal energy 
\begin{align*}
\mathcal{J}_\varepsilon(v_h) \coloneqq \int_\Omega \kappa_\varepsilon(|\nabla v_h |)\dx - \int_\Omega fv_h\dx\qquad\text{for all }v_h\in S_0^1(\tria).
\end{align*}
Let the minimizer of the regularized primal and dual energy read
\begin{align*}
u_{h,\varepsilon} = \argmin_{v_h \in S_0^1(\tria)} \mathcal{J}_\varepsilon(v_h)\qquad\text{and}\qquad \sigma_{h,\varepsilon} \in \argmin_{\tau_h \in\mathbb{P}_0(\textup{div}_h{=}{\shortminus}f,\tria)} \mathcal{J}_\varepsilon^*(\tau_h).
\end{align*}
An application of duality theory shows that 
\begin{align}\label{eq:discDualityGap}
-  \mathcal{J}^*_\varepsilon(\sigma_{h,\varepsilon}) = \mathcal{J}_\varepsilon(u_{h,\varepsilon}).
\end{align}
Moreover, the (discrete) minimizers  are related by the identities
\begin{align*}
\frac{\varphi_\varepsilon'(|\nabla u_{h,\varepsilon}|)}{|\nabla u_{h,\varepsilon}|} \nabla u_{h,\varepsilon} = \sigma_{h,\varepsilon}\qquad\text{and}\qquad \frac{(\varphi^*_\varepsilon)'(|\sigma_{h,\varepsilon}|)}{|\sigma_{h,\varepsilon}|} \sigma_{h,\varepsilon} = \nabla u_{h,\varepsilon}.
\end{align*}
As in~\cite{DieningFornasierWank17,DieningEttwein08} we define for all $s,t\geq 0$ the shifted N-functions
\begin{align*}
\varphi_{\varepsilon,t}(s) \coloneqq \int_0^s \frac{\varphi_\varepsilon'(t \vee \tau)}{t \vee \tau}\tau  \,\mathrm{d}\tau\qquad\text{and}\qquad
\varphi_{\varepsilon,t}^*(s) \coloneqq \int_0^s \frac{(\phi_\varepsilon^*)'(t \vee \tau)}{t \vee \tau}\tau  \,\mathrm{d}\tau.
\end{align*}
We further introduce the quantity 
\begin{align*}
V_\varepsilon(P) \coloneqq \begin{cases}
\sqrt{\frac{\phi_\varepsilon'(|P|)}{|P|}}P&\text{for }P\neq 0,\\
0&\text{for }P=0.
\end{cases}
\end{align*}
Let $\sigma_{h,n} \in \mathbb{P}_0(\textup{div}_h{=}{\shortminus}f,\tria)$ and $u_{h,n}\in S_0^1(\tria)$ be computed with the iterative scheme in \eqref{eq:PrimalKacanov} and let $\rho>0$ be some fixed weight.
The error indicator that indicates errors caused by
\begin{enumerate}
\item the upper interval bound $\varepsilon_+$ reads
$\eta_{\varepsilon^+}^2(\sigma_{h,n}) \coloneqq \mathcal{J}^*_{\varepsilon}(\sigma_{h,n}) - \mathcal{J}^*_{(\varepsilon_-,\infty)}(\sigma_{h,n})$,\label{item:1}
\item the lower interval bound $\varepsilon_-$ reads
$
\eta_{\varepsilon^-}^2(\sigma_{h,n}) \coloneqq \mathcal{J}^*_{\varepsilon}(\sigma_{h,n}) - \mathcal{J}^*_{(0,\varepsilon_+)}(\sigma_{h,n})$,\label{item:2}
\item the error due to the fixed-point iteration reads \label{item:3}
\begin{align*}
\eta^2_{\textup{Ka\v{c}},\varepsilon}(\sigma_{h,n}) \coloneqq \mathcal{J}_\varepsilon(u_{h,n}) + \mathcal{J}^*_\varepsilon(\sigma_{h,n}),
\end{align*}
\item the discretization reads $\eta_{h,\varepsilon}^2(u_n) \coloneqq \rho\, \sum_{T\in \mathcal{T}} \eta^2_{h,\varepsilon}(u_n,T)$ with ($\mathcal{F}(T)$ denotes the set of all faces of $T\in \tria$ and $h_\gamma$ denotes the diameter of $\gamma \in \mathcal{F}(T)$) \label{item:4}
\begin{align*}
\qquad\qquad\eta^2_{h,\varepsilon}(u_{h,n},T) & \coloneqq \int_T (\varphi_{\varepsilon, |\nabla u_{h,n}|})^* (h_T |f|) + \sum_{\gamma \in \mathcal{F}(T)} h_\gamma \int_\gamma|\llbracket V_\varepsilon(\nabla u_{h,n}) \rrbracket_\gamma|^2\,\mathrm{d}s.
\end{align*}
\end{enumerate}
The error indicators in \ref{item:1}--\ref{item:2} are motivated by the definition of the relaxation, in \ref{item:3} by the identity in \eqref{eq:discDualityGap}, and in \ref{item:4} by the a posteriori error control for the primal formulation of the $p$-Laplace problem \cite{BelenkiDieningKreuzer12,DieningKreuzer08}.
%
The error indicators lead to the following numerical scheme.

\begin{figure}[H]
\begin{algorithm}[H]
\SetAlgoLined
\TitleOfAlgo{Adaptive relaxed $p$-Ka{\v c}anov{} algorithm}
\KwData{Given: $f \in L^q(\Omega)$, $\sigma_0 \in
  \Sigma_h$;}
\KwResult{Approximations of the solutions to \eqref{eq:PrimalProblemIntro} and \eqref{eq:dualProblemIntro};}
Initialize: $\varepsilon_1=[1,1] \subset (0,\infty)$,
$n=0$, $\sigma_{h,0} = 0$\;

\While{desired accuracy is not achieved yet}{
  Increase $n$ by 1 and calculate $u_{h,n}$ and $\sigma_{h,n}$ by means of \eqref{eq:PrimalKacanov}\;
  Calculate error indicators
  $\eta^2_{\varepsilon^+}(\sigma_{h,n})$, $\eta^2_{\varepsilon^-}(\sigma_{h,n})$,
  $\eta^2_{h,\varepsilon}(u_{h,n})$, $\eta^2_{\text{Ka{\v c}},\varepsilon}(\sigma_{h,n})$\;
\uIf{$\eta^2_{\varepsilon^+}(\sigma_{h,n})$ is the largest}{
Set $\varepsilon_{n+1,+}
  \coloneqq 1.25 \cdot \varepsilon_{n,+}$ and $\varepsilon_{n+1,-}
  \coloneqq \varepsilon_{n,-}$\;
}
\uElseIf{$\eta^2_{\varepsilon^-}(\sigma_{h,n})$ is the largest}{
Set $\varepsilon_{n+1,-}
  \coloneqq 0.8 \cdot \varepsilon_{n,-}$ and $\varepsilon_{n+1,+}
  \coloneqq \varepsilon_{n,+}$\;
}
\uElseIf{$\eta^2_h(u_{h,n})$ is the largest}
{
Refine mesh with D\"orfler marking and set $\varepsilon_{n+1} \coloneqq \varepsilon_n$\;}
\uElseIf{$\eta^2_{\textup{Ka\v{c}},\varepsilon}(\sigma_{h,n})$ is the largest}{
Set $\varepsilon_{n+1} \coloneqq \varepsilon_n$\;
}}
\end{algorithm}
\caption{Adaptive relaxed $p$-Ka{\v c}anov{} algorithm}\label{fig:Algo}
\end{figure}
\section{Numerical experiments}\label{sec:NumExp}
We implemented the numerical scheme suggested in Section~\ref{sec:Discretization} with the open source tool for solving partial differential equations FEniCS \cite{LoggMardalWells12}. The supplementary material of this paper contains the implementation.
\subsection{Convergence for fixed mesh}\label{subsec:Exp1}
Our first experiment emphasizes the advantages of the adaptive relaxed $p$-Ka{\v c}anov{} algorithm in Figure~\ref{fig:Algo} compared to the strategy suggested in Section~\ref{sec:AlgebraicRate}. To allow for a better comparison we do not refine the mesh $\tria$, which is a partition of the unit circle with about $10^5$ vertices.  Our right-hand side $f=1$ is constant and $p=10$. The initial relaxation interval reads $\varepsilon = [1,1]$.  In a first computation we use the strategy suggested in Section~\ref{sec:AlgebraicRate} with $\alpha = \beta = (2-q)^{-1}/2$. Figure~\ref{fig:Exp1} displays the resulting algebraic rate of convergence of the energies $\mathcal{J}_\varepsilon(u_{h,n})$ and $\mathcal{J}_\varepsilon^*(\sigma_{h,n})$ of the iterates towards the energy $\mathcal{J}_\varepsilon(u_{h})$ of (a very accurate approximation of) the exact discrete minimizer. The relative error of the approximated minimal energy and the exact minimal energy is about $3\times 10^{-5}$. In a second computation we adapt the relaxation interval according adaptive relaxed $p$-Ka{\v c}anov{} algorithm in Figure~\ref{fig:Algo} (without mesh refinements). This strategy increases the rate of convergence significantly.  The first 22 iterations cause the decrease of $\varepsilon_-$.  Thereafter, it takes more and more \Kacanov{} iterations to close the dual gap, which leads to the zick-zack pattern in Figure~\ref{fig:Exp1}.  This behavior is in agreement with the theory, which states that speed of convergence of the \Kacanov{} iterations depends on the relaxation interval.
\begin{figure}
{\centering
\begin{tikzpicture}
\begin{axis}[
clip=false,
width=.6\textwidth,
height=.45\textwidth,
ymode = log,
cycle multi list={\nextlist MyColors},
scale = {1},
xlabel={Iterations},
clip = true,
legend cell align=left,
legend style={legend columns=1,legend pos= outer north east,font=\fontsize{7}{5}\selectfont}
]
	\addplot table [x=iter,y=ErrorPrimal] {Data/Experiment1.txt};
	\addplot table [x=iter,y=errorDual] {Data/Experiment1.txt};
	
	\addplot table [x=iter,y=ErrorPrimal] {Data/Experiment1_adapt.txt};
	\addplot table [x=iter,y=errorDual] {Data/Experiment1_adapt.txt};
	\legend{
	{$\mathcal{J}_\varepsilon(u_{h,n}) - \mathcal{J}(u_h)$ },{$\mathcal{J}^*_\varepsilon(\sigma_{h,n}) - \mathcal{J}^*(\sigma_h)$ },{$\mathcal{J}_\varepsilon(u_{h,n}) - \mathcal{J}(u_h)$ (adaptive)},{$\mathcal{J}^*_\varepsilon(\sigma_{h,n}) - \mathcal{J}^*(\sigma_h)$ (adaptive)}};
\end{axis}
\end{tikzpicture}
\caption{Convergence of the energies in Experiment 1 towards the energy of the exact discrete minimizer.} \label{fig:Exp1}}
\end{figure}
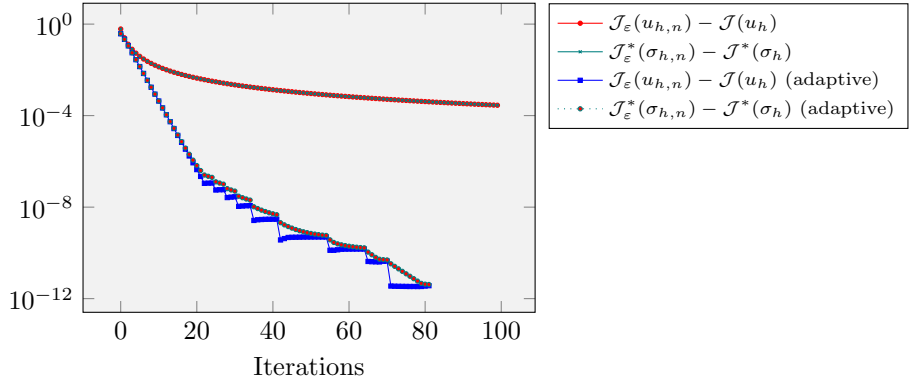
\subsection{Fully adaptive scheme}\label{subsec:Exp2}
In our second experiment we illustrate the performance of the adaptive scheme by solving the $p$-Laplacian with right-hand side $f=2$ and large exponent $p=100$ on the L-shaped domain $\Omega = [-1,1]^2\setminus [0,1)^2$. We utilize the adaptive refinement strategy from Section~\ref{sec:adaptScheme} with weight $\rho = 10^{-3}$. 
Figure~\ref{fig:Exp2} displays the convergence history plot. 
The left-hand side displays the errors against the accumulated number of degrees of freedom, that is, we sum up all degrees of freedom in each iteration of the adaptive loop. The right-hand side plots the errors before each mesh refinement against the degrees of freedom of the current iterate.

In both cases the rate of convergence seems to be slightly worse than $-2/3$.
This rate might be caused by the (pre-asymptotic) effect that the adaptive scheme causes line refinements, see Figure~\ref{fig:AdaptMesh}.
Moreover, it is unclear if the scheme leads to an optimal rate since the proofs of optimal convergence in \cite{BelenkiDieningKreuzer12,DieningKreuzer08} use the exact discrete solution of the non-linear and not regularized problem.

Moreover, we observe huge gaps $\mathcal{J}_\varepsilon(u_{h,n}) + \mathcal{J}_\varepsilon^*(\sigma_{h,n})$ after refinements of the mesh $\tria$ or enlargements of the relaxation interval $\varepsilon$, see~ Figure~\ref{fig:Exp2}. The large error in the primal energies indicates that these gaps are mainly caused by the fact that $u_{h,n}$ is a bad approximation of the exact discrete minimizer of the primal problem. Thus, it might be advantageous to use an error estimator that estimates the dual error. Alternatively, one might use an error estimator that is efficient and reliable for any discrete approximation.




\begin{figure}
{\centering
\begin{tikzpicture}
\begin{axis}[
clip=false,
width=.54\textwidth,
height=.45\textwidth,
ymode = log,
xmode = log,
cycle multi list={\nextlist MyColors},
scale = {1},
xlabel={Accumulated ndofs},
clip = true,
legend cell align=left,
legend style={legend columns=1,legend pos= outer north east,font=\fontsize{7}{5}\selectfont}
]
	\addplot table [x=ndofSum,y=eta_h] {Data/Experiment3_d2_New.txt};\label{line:1b}
	\addplot table [x=ndofSum,y=J_eps_dist] {Data/Experiment3_d2_New.txt};\label{line:2b}
	\addplot table [x=ndofSum,y=J_eps_star_dist] {Data/Experiment3_d2_New.txt};\label{line:3b}
	\addplot[dashed, sharp plot,update limits=false] coordinates {(1e5,5e-4) (1e8,5e-6)};
\end{axis}
\end{tikzpicture}
\begin{tikzpicture}
\begin{axis}[
clip=false,
width=.54\textwidth,
height=.45\textwidth,
ymode = log,
xmode = log,
cycle multi list={\nextlist MyColors},
scale = {1},
xlabel={ndofs},
clip = true,
legend cell align=left,
legend style={legend columns=1,legend pos= outer north east,font=\fontsize{7}{5}\selectfont}
]
	\addplot table [x=ndof,y=eta_h] {Data/Exp3_refs_d2.txt};
	\addplot table [x=ndof,y=J_eps_dist] {Data/Exp3_refs_d2.txt};
	\addplot table [x=ndof,y=J_eps_star_dist] {Data/Exp3_refs_d2.txt};
	\addplot[dashed, sharp plot,update limits=false] coordinates {(1e3,1e-3) (1e6,1e-5)};
\end{axis}
\end{tikzpicture}
\caption{Convergence history of $\eta_h^2(u_{h,n})$ (\ref{line:1b}), $|\mathcal{J}_{\varepsilon_n}(u_{h,n}) - \mathcal{J}(u)|$ (\ref{line:2b}), $|\mathcal{J}_{\varepsilon_n}^*(\sigma_{h,n}) - \mathcal{J}^*(\sigma)|$ (\ref{line:3b}). The dashed line indicates the slope $-2/3$.} \label{fig:Exp2}}
\end{figure}
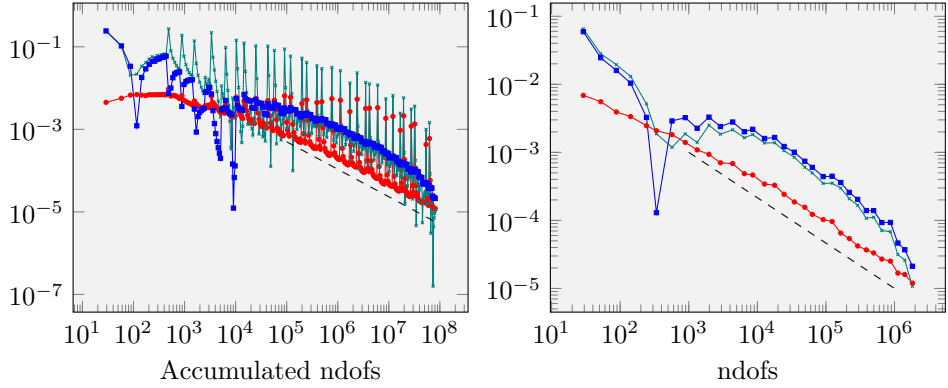

\begin{figure}
\includegraphics[scale=.2]{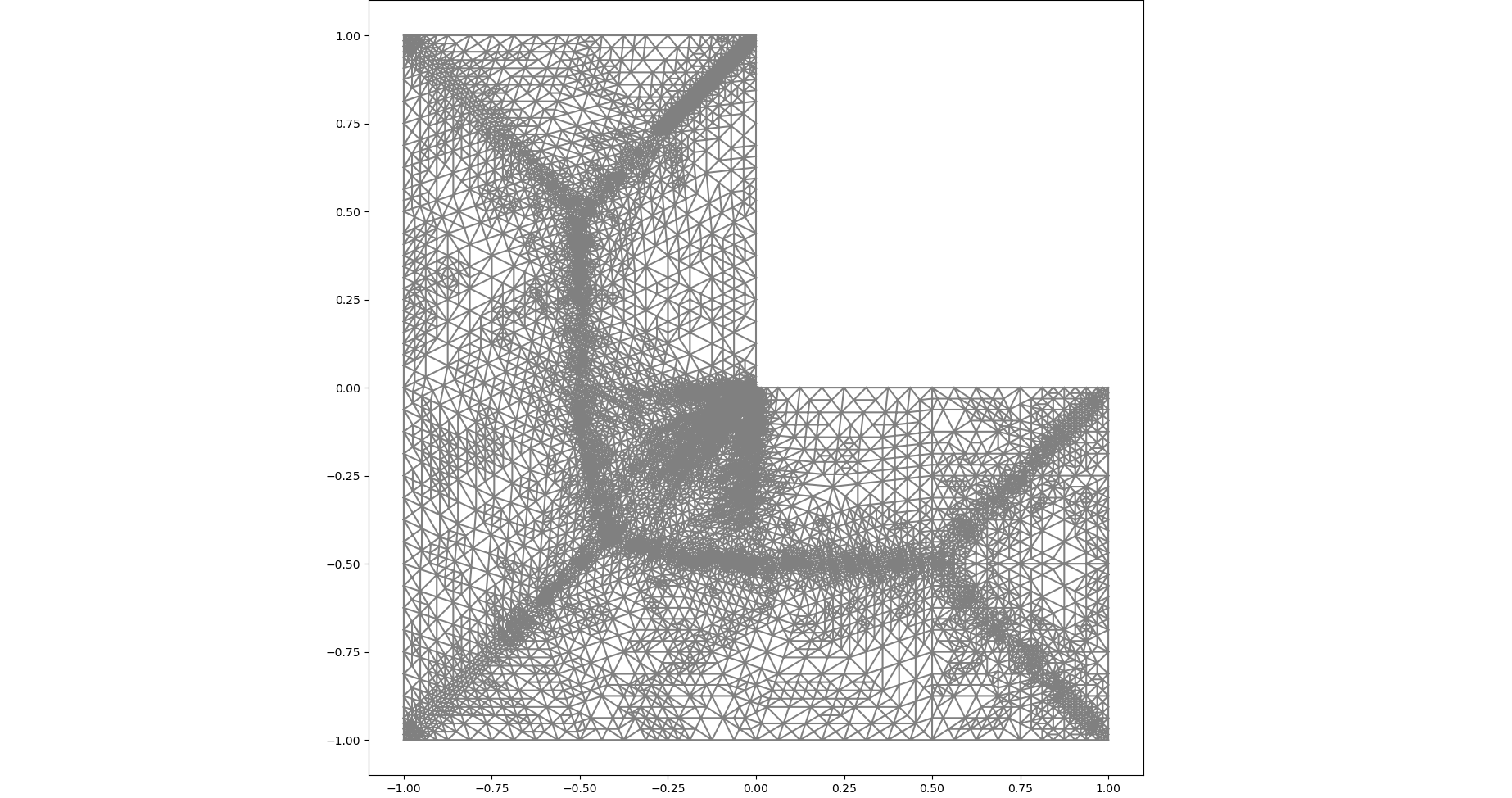}
\caption{Adaptively refined mesh in Experiment 2 with about $10^4$ vertices.}\label{fig:AdaptMesh}
\end{figure}
\subsection{Comparison to steepest descent}\label{subsec:Exp3}
This experiment compares our scheme with the steepest descent approach suggested in \cite{HuangLiLiu07}.
Given an approximation $u_{h,n}^\textup{steep} \in S_0^1(\tria)$ and a small regularization parameter $0<\delta \ll 1$, this scheme computes a descend direction $d_{h,n} \in S_0^1(\tria)$ with
\begin{align}\label{eq:SteepestDescent}
\begin{aligned}
&\int_\Omega (\delta + |\nabla u_{h,n}|)^{p-2}\nabla d_{h,n} \cdot v_h\dx \\
&\qquad = - \int_\Omega   |\nabla u_{h,n}|^{p-2}\nabla u^\textup{steep}_{h,n} \cdot \nabla v_h \dx + \int_\Omega fv_h\dx\quad\text{for all }v_h \in S_0^1(\tria).
\end{aligned}
\end{align} 
Thereafter, it performs a line search to determine
\begin{align}\label{eq:lineSearch}
u^\textup{steep}_{h,n+1} \coloneqq u_{h,n}^\textup{steep} + \alpha_n d_{h,n} \qquad \text{with}\qquad
\alpha_n \coloneqq \argmin_{\alpha \geq 0} \mathcal{J}(u_{h,n}^\textup{steep} + \alpha d_{h,n}). 
\end{align}
Notice that the line search is an additional costly effort. We compare the regularized \Kacanov{} scheme and the steepest descent approach for different values $p \in \lbrace 5,10,20,50\rbrace$. The right-hand side is the constant function $f=2$.
The underlying domain $\Omega = (-1,1)^2\setminus [0,1)$ is the L-shaped domain and we use meshes $\tria$ with about $10^3$ and $10^5$ vertices.
The meshes result from an adaptive mesh refinement strategy similar to the one in Experiment~2. Our initial iterate on these meshes is the Galerkin approximation $u_{h,0} = u^\textup{steep}_{h,0} \in S_0^1(\tria)$ to the Poisson model problem (that is $p=2$). We fix our relaxation interval $\varepsilon = (10^{-6},10^6)$ and the relaxation parameter $\delta = 10^{-6}$.
Notice that we do in general not recommend such large relaxation intervals, since it slows down the rate of convergence (see Experiment 1--2). However, we do not want to benefit in our comparison from improved rates due to our relaxation/regularization. 
We compute the iterates of the \Kacanov{} and steepest descent schemes until the duality gap of the \Kacanov{} scheme is below
\begin{align*}
\mathcal{J}_\varepsilon(u_{h,n}) + \mathcal{J}_\varepsilon^*(\sigma_{h,n}) \leq 10^{-7}.
\end{align*}
Thereafter, we proceed with the \Kacanov{} scheme until the duality gap is below $10^{-9}$ to compute a reference value for the exact minimizer $\mathcal{J}(u_{h})$.

Figure~\ref{fig:Experiment3} displays the resulting convergence history plots. It shows that the energy differences $\mathcal{J}^*(\sigma_{h,n}) + \mathcal{J}(u_h)$ in the dual energies decrease monotonically. After a very fast decay in a pre-asymptotic regime, we see the expected exponential rate of convergence. The rate slows down as $p$ increases but seems to be robust with respect to mesh refinements. 
The primal energy differences $\mathcal{J}(u_{h,n}) - \mathcal{J}(u_h)$ show some oscillatory behavior in a pre-asymptotic regime. After that regime (which seems to be larger than the pre-asymptotic regime for the dual energy differences) we see an exponential convergence with a rate that is slightly better than the rate of the dual energy errors. 
The steepest descent approach seems to converge on coarse meshes asymptotically faster than the \Kacanov{} scheme. However, the size of the pre-asymptotic regime depends not only on $p$ (as the size of the pre-asymptotic regimes for the \Kacanov{} scheme does as well) but also on the mesh. In fact, our calculations never overcome these regimes for the meshes with about $10^5$ degrees of freedom. Moreover, for $p=50$ and the fine mesh $\tria$ the steepest descent direction $d_{h,n} \in S_0^1(\tria)$ does not seem to be a proper descent direction, that is, our line search in \eqref{eq:lineSearch} does not find a minimizer $\alpha_n > 0$.
This indicates difficulties for large values of $p$ caused by numerical instabilities and the regularization $\delta$.

Overall, the experiment shows that the (regularized) \Kacanov{} iterations outperform the steepest descent scheme. The \Kacanov{} scheme
\begin{itemize}
\item is much faster since it avoids an involved line search,
\item shows significantly better pre-asymptotic behavior, in particular for large values of $p$,
\item is robust with respect to (adaptive) mesh refinements.
\end{itemize}
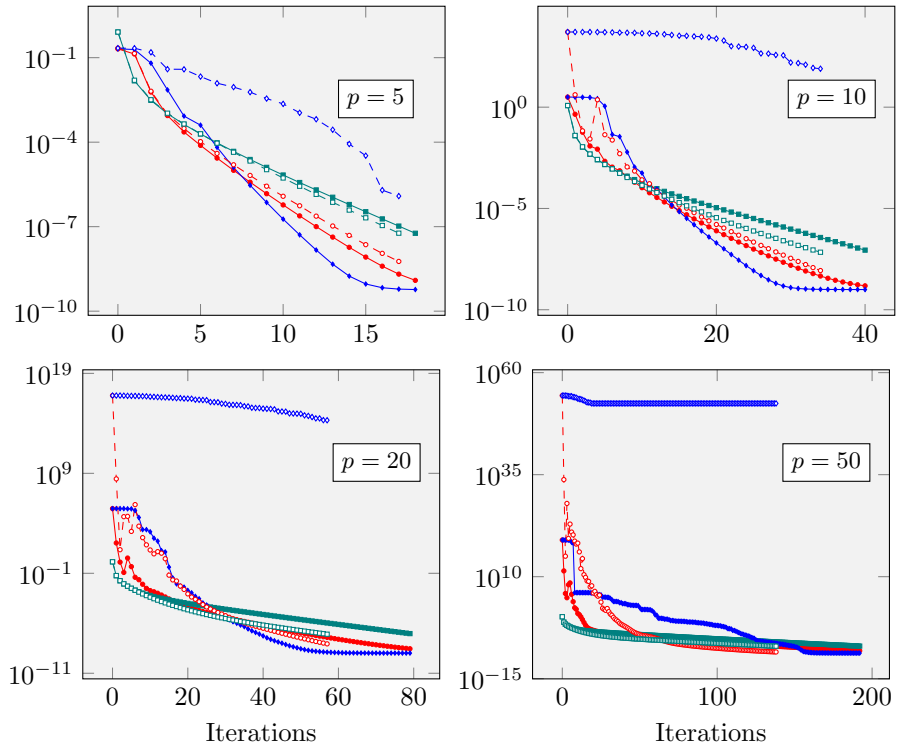
\begin{figure}
\begin{tikzpicture}
\begin{axis}[
clip=false,
width=.5\textwidth,
height=.45\textwidth,
ymode = log,
cycle multi list={\nextlist MyColors3},
scale = {1},
clip = true,
]
	\addplot table [x=iter,y=EnergyDifKac] {Data/DataExp3/Experiment4_p_is_5ndof1128.txt};
	\addplot table [x=iter,y=EnergyDifKacDual] {Data/DataExp3/Experiment4_p_is_5ndof1128.txt};
	\addplot table [x=iter,y=EnergyDifSteep] {Data/DataExp3/Experiment4_p_is_5ndof1128.txt};

	\addplot table [x=iter,y=EnergyDifKac] {Data/DataExp3/Experiment4_p_is_5ndof122107.txt};
	\addplot table [x=iter,y=EnergyDifKacDual] {Data/DataExp3/Experiment4_p_is_5ndof122107.txt};
	\addplot table [x=iter,y=EnergyDifSteep] {Data/DataExp3/Experiment4_p_is_5ndof122107.txt};		
			\node[draw,fill=white, anchor=west] at (rel axis cs: .7,.7) {\small $p = 5$};
			\end{axis}
\end{tikzpicture}
\begin{tikzpicture}
\begin{axis}[
clip=false,
width=.5\textwidth,
height=.45\textwidth,
ymode = log,
cycle multi list={\nextlist MyColors3},
scale = {1},
clip = true,
]
	\addplot table [x=iter,y=EnergyDifKac] {Data/DataExp3/Experiment4_p_is_10ndof1360.txt};
	\addplot table [x=iter,y=EnergyDifKacDual] {Data/DataExp3/Experiment4_p_is_10ndof1360.txt};
	\addplot table [x=iter,y=EnergyDifSteep] {Data/DataExp3/Experiment4_p_is_10ndof1360.txt};
	\addplot table [x=iter,y=EnergyDifKac] {Data/DataExp3/Experiment4_p_is_10ndof110621.txt};
	\addplot table [x=iter,y=EnergyDifKacDual] {Data/DataExp3/Experiment4_p_is_10ndof110621.txt};
	\addplot table [x=iter,y=EnergyDifSteep] {Data/DataExp3/Experiment4_p_is_10ndof110621.txt};
			\node[draw,fill=white, anchor=west] at (rel axis cs: .7,.7) {\small $p = 10$};
			\end{axis}
\end{tikzpicture}
\begin{tikzpicture}
\begin{axis}[
clip=false,
width=.5\textwidth,
height=.45\textwidth,
ymode = log,
cycle multi list={\nextlist MyColors3},
scale = {1},
xlabel={Iterations},
clip = true,
]
	\addplot table [x=iter,y=EnergyDifKac] {Data/DataExp3/Experiment4_p_is_20ndof1227.txt};
	\addplot table [x=iter,y=EnergyDifKacDual] {Data/DataExp3/Experiment4_p_is_20ndof1227.txt};
	\addplot table [x=iter,y=EnergyDifSteep] {Data/DataExp3/Experiment4_p_is_20ndof1227.txt};
	\addplot table [x=iter,y=EnergyDifKac] {Data/DataExp3/Experiment4_p_is_20ndof129614.txt};
	\addplot table [x=iter,y=EnergyDifKacDual] {Data/DataExp3/Experiment4_p_is_20ndof129614.txt};
	\addplot table [x=iter,y=EnergyDifSteep] {Data/DataExp3/Experiment4_p_is_20ndof129614.txt};		
				\node[draw,fill=white, anchor=west] at (rel axis cs: .7,.7) {\small $p = 20$};
\end{axis}
\end{tikzpicture}
\begin{tikzpicture}
\begin{axis}[
clip=false,
width=.5\textwidth,
height=.45\textwidth,
ymode = log,
cycle multi list={\nextlist MyColors3},
scale = {1},
xlabel={Iterations},
clip = true,
legend cell align=left,
legend style={legend columns=3,font=\fontsize{7}{5}\selectfont}
]
	\addplot table [x=iter,y=EnergyDifKac] {Data/DataExp3/Experiment4_p_is_50ndof1184.txt};\label{line:1}
	\addplot table [x=iter,y=EnergyDifKacDual] {Data/DataExp3/Experiment4_p_is_50ndof1184.txt};\label{line:2}
	\addplot table [x=iter,y=EnergyDifSteep] {Data/DataExp3/Experiment4_p_is_50ndof1184.txt};\label{line:3}

	\addplot table [x=iter,y=EnergyDifKac] {Data/DataExp3/Experiment4_p_is_50ndof119757.txt};
	\addplot table [x=iter,y=EnergyDifKacDual] {Data/DataExp3/Experiment4_p_is_50ndof119757.txt};
	\addplot table [x=iter,y=EnergyDifSteep] {Data/DataExp3/Experiment4_p_is_50ndof119757.txt};	
			\node[draw,fill=white, anchor=west] at (rel axis cs: .7,.7) {\small $p = 50$};

\end{axis}
\end{tikzpicture}
\caption{Convergence of $\mathcal{J}(u_{h,n}) - \mathcal{J}(u_h)$ (\ref{line:1}), $\mathcal{J}^*(\sigma_{h,n}) + \mathcal{J}(u_h)$ (\ref{line:2}), and $\mathcal{J}(u^\textup{steep}_{h,n}) - \mathcal{J}(u_h)$ (\ref{line:3}) in Experiment 3 for various $p$. The lines with filled markers display the results for meshes with about $10^3$ vertices and the lines with markers filled white displays the results for meshes with about $10^5$ vertices.}\label{fig:Experiment3}
\end{figure}
\printbibliography


\end{document}